\documentclass[12pt]{amsart}

\usepackage{amssymb,latexsym,mathrsfs,amsmath,amsthm}
\usepackage{amsfonts,verbatim}
\newtheorem{theorem}{Theorem}[section]
\newtheorem{proposition}[theorem]{Proposition}

\theoremstyle{remark} \newtheorem{remark}[theorem]{Remark}

\newcommand\cis{\operatorname{cis}}

\begin{document}

\title[Non-Friedrichs Self-adjoint Extensions of the Laplacian in $\mathbb{R}^d$]{Non-Friedrichs Self-adjoint extensions of the Laplacian in $\mathbb{R}^d$}
\author{Paul Lin}
\address{Paul Lin, Department of Mathematics, Australian National University, 
ACT 0200 Australia
}
\email{Paul.Lin@anu.edu.au}

%\subjclass{42B20 (primary), 47F05, 58J05 (secondary).}

\begin{abstract}
This is an expository paper about self-adjoint extensions of the Laplacian on $\mathbb{R}^d$, initially defined on functions supported away from a point.
Let $L$ be the Laplacian $-\sum_{i=1}^d\frac{\partial^2}{\partial x_i^2}$ with domain $C_c^\infty(\mathbb{R}^d\backslash\{0\})$.
We determine all the self-adjoint extensions of $L$ by calculating the deficiency subspaces of the closure $\overline{L}$ of $L$. 
We give the details of the calculations here since, to my knowledge, they are nowhere else in the literature.
In the case $d=3$,  the self-adjoint extensions $L_\theta$ are parametrized by a circle $\theta\in[-\pi, \pi)$, 
or equivalently by the real line plus a point $\mu\in\mathbb{R}\cup\{\infty\}$. 
We show that the non-Friedrichs extensions have either a single eigenvalue or a single resonance.
This is of some interest since self-adjoint Schr{\"o}dinger operators 
$\Delta+V$ on $\mathbb{R}^d$ either have no resonance (if $V\equiv 0$) or infinitely many (if $V\not\equiv 0$), see \cite[Prop. 4.3]{RBM}.
For these operators $L^\mu$, the kernel of $(L^\mu-\lambda^2)^{-1}$ is calculated explicitly, see \cite{AGHH}. 
Finally we use this knowledge to study a wave equation involving $L^\mu$, 
and explicitly determine the wave kernel.
\end{abstract}

\maketitle

\section{The operator $L$ and its closure}

\noindent We work in the Hilbert space $L^2(\mathbb{R}^d)$, and start with the unbounded operator $L=-\sum_{i=1}^d\frac{\partial^2}{\partial x_i^2}$  
with domain $D(L)=C_c^\infty(\mathbb{R}^d\backslash\{0\})$. 
Note that $L$ is symmetric, which can be shown by integration by parts, 
but it is not self-adjoint as the domain of its adjoint  $D(L^*)$ contains $C_c^\infty(\mathbb{R}^d)$ which is strictly bigger than $D(L)$. 
Recall that the deficiency subspaces $\mathscr{K}_+$, $\mathscr{K}_-$ of $L$ are the 
null spaces of the operators $i-L^*$, $i+L^*$ respectively.
By applying von Neumann's theorem to $L$, 
we know that the two deficiency subspaces of $L$ have the same dimension hence $L$ has self-adjoint extensions. 
To find out how many self-adjoint extensions $L$ has and what they look like, we use the following proposition
due to von Neumann, for the proof see \cite[Sec. X.1]{RS}:

\begin{proposition}
Let $A$ be a closed symmetric operator with equal deficiency indices. Then 
there is a one-one correspondence between self-adjoint extensions of $A$ and unitary maps from $\mathscr{K}_+$ onto $\mathscr{K}_-$. 
If $U$ is such a unitary map, the corresponding self-adjoint extension $A_U$ has domain
\[D(A_U)=\{\varphi+\psi+U\psi:\varphi\in D(A),\psi\in\mathscr{K}_+\}\]
and 
\[A_U(\varphi+\psi+U\psi)=A\varphi+i\psi-iU\psi.\]
\end{proposition}

\vspace{3mm}

\noindent The above proposition requires a closed symmetric operator.
To find the domain of $\overline{L}$ means to complete $C_c^\infty(\mathbb{R}^d\backslash\{0\})$
under the norm $||\cdot||+||L(\cdot)||$. 
(In this paper, the norm notation $||\cdot||$ always denotes the $L^2$-norm.)
This norm is equivalent to the $W^{2,2}$-norm hence we have 
$D(\overline{L})=W_0^{2,2}(\mathbb{R}^d\backslash\{0\})\subseteq W^{2,2}(\mathbb{R}^d)$.
The situation differs as dimension varies. We first treat the case $d>4$. 

\subsection{Case d$>$4}

\begin{proposition}
For $\mathbb{R}^d$ with $d>4$, we have 
\[D(\overline{L})=W^{2,2}(\mathbb{R}^d).\]
\end{proposition}
\begin{proof}
From above, we already know that $D(\overline{L})\subseteq W^{2,2}(\mathbb{R}^d)$. 
Since $C_c^\infty(\mathbb{R}^d)$ is dense in $W^{2,2}(\mathbb{R}^d)$, 
we pick any $\psi\in C_c^\infty(\mathbb{R}^d)$, 
and approximate it with a sequence of functions in $C_c^\infty(\mathbb{R}^d\backslash\{0\})$ 
that converges to  $\psi$ under the $W^{2,2}$-norm.
Let $\varphi$ be a smooth function such that $\varphi=1$ on $B_{\frac{1}{2}}(0)$ and $\varphi=0$ outside $B_1(0)$. 
Then define $\varphi_\epsilon(x)=1-\varphi(\frac{x}{\epsilon})$. 
We now show that $\psi\varphi_\epsilon$ converges to $\psi$ under the $W^{2,2}$-norm when $\epsilon$ approaches $0$. 
Indeed, $\psi\varphi_\epsilon\in C_c^\infty(\mathbb{R}^d\backslash\{0\})$, and
\[||\psi\varphi_\epsilon-\psi||^2
=\int_{\mathbb{R}^d}|\psi(x)\varphi_(\frac{x}{\epsilon})|^2dx\leq M_1^2\int_{\mathbb{R}^d}|\varphi(\frac{x}{\epsilon})|^2dx
=\epsilon^d M_1^2||\varphi||^2\rightarrow 0, \mbox{ as }\epsilon\rightarrow 0,\]
where $M_1$ is the maximum value of $\psi$. Now consider $||L(\psi\varphi_\epsilon-\psi))||$, product rule creates three terms,
\[||(L\psi)(\varphi_\epsilon-1)||^2=\int_{\mathbb{R}^d}|(L\psi)(x)\varphi(\frac{x}{\epsilon})|^2dx
\leq M_2^2\int_{\mathbb{R}^d}|\varphi(\frac{x}{\epsilon})|^2dx
=\epsilon^dM_2^2||\varphi||^2\rightarrow 0,\mbox{ as }\epsilon\rightarrow 0,\]
where $M_2$ is the maximum value of $L\psi$. Since $d-2>0$,
\begin{equation*}
\begin{split}
||\nabla\psi\cdot\nabla(\varphi_\epsilon-1)||^2
\leq M_3^2\int_{\mathbb{R}^d}|\nabla\varphi(\frac{x}{\epsilon})|^2dx
=\epsilon^{-2}M_3^2&\int_{\mathbb{R}^d}|(\nabla\varphi)(\frac{x}{\epsilon})|^2dx\\
&=\epsilon^{d-2} M_3^2|||\nabla\varphi|||^2\rightarrow 0\mbox{ as }\epsilon\rightarrow 0,
\end{split}
\end{equation*}
the partial derivatives of $\psi$ take values smaller than $M_3$. At last, since $d-4>0$, 
\begin{equation*}
\begin{split}
||\psi L(\varphi_\epsilon-1)||^2\leq M_1^2\int_{\mathbb{R}^d}|L\varphi(\frac{x}{\epsilon})|^2dx
=M_1^2\epsilon^{-4}&\int_{\mathbb{R}^d}|(L\varphi)(\frac{x}{\epsilon})|^2dx\\
&=M_1^2\epsilon^{d-4}||L\varphi||^2\rightarrow 0\mbox{ as }\epsilon\rightarrow 0.
\end{split}
\end{equation*}
We have established in the case $d>4$, 
\[D(\overline{L})=W^{2,2}(\mathbb{R}^d).\]
\end{proof}

\subsection{Case d=4}

\noindent We can see for the case $d\leq4$, the above proof fails because $M_1^2\epsilon^{d-4}||L\varphi||^2$ doesn't converge to $0$. 
In fact, when $d<4$, the domain of $\overline{L}$ is smaller than $W^{2,2}(\mathbb{R}^d)$. 
For the case $d=4$, we still have $D(\overline{L})=W^{2,2}(\mathbb{R}^4)$, but we can't use the same $\varphi_\epsilon$ as defined in the above proof 
because $||L(\varphi_\epsilon-1)||=||L\varphi(\frac{x}{\epsilon})||$ is independent of $\epsilon$. So to get convergence 
we need to define $\varphi_\epsilon$ in a way that breaks the scaling invariance of $||L(\varphi_\epsilon-1)||$.

\begin{proposition}
For $\mathbb{R}^4$, we have 
\[D(\overline{L})=W^{2,2}(\mathbb{R}^4).\]
\end{proposition}
\begin{proof}
As before, we pick any arbitrary $\psi\in C_c^\infty(\mathbb{R}^4)$. 
Let $\varphi:[0,\infty)\rightarrow[0,1]$ be a smooth function such that such that $\varphi\big([0,\frac{1}{2}]\big)=1$ and $\varphi\big([1,\infty)\big)=0$. 
Then define $\varphi_\epsilon(x):=1-\varphi\big((\frac{|x|}{\epsilon})^{\epsilon}\big)$.
We now verify that as $\epsilon\rightarrow 0$, $\psi\varphi_\epsilon$ approximates $\psi$ under the $W^{2,2}$-norm.
Indeed, $\psi\varphi_\epsilon\in C_c^\infty(\mathbb{R}^4\backslash\{0\})$, and
\[||\psi\varphi_\epsilon-\psi||^2=\int_{\mathbb{R}^4}|\psi(x)\varphi\big((\frac{|x|}{\epsilon})^{\epsilon}\big)|^2dx.\]
We show that the above norm tends to $0$ when $n$ becomes large.
Consider any $x$, when $\epsilon\leq |x|$, we have $(\frac{|x|}{\epsilon})^{\epsilon}\geq 1$, 
hence $|\psi(x)\varphi\big((\frac{|x|}{\epsilon})^{\epsilon}\big)|^2\rightarrow 0$ pointwise. 
Moreover, $|\psi(x)\varphi\big((\frac{|x|}{\epsilon})^{\epsilon}\big)|^2$ is bounded above by the $L^1$ function $|\psi|^2$, so we 
can conclude  $||\psi\varphi_\epsilon-\psi||\rightarrow 0$ as $\epsilon\rightarrow 0$.

\vspace{2mm}

\noindent Again as before, the term $||L(\psi\varphi_\epsilon-\psi)||$ produces three terms: 
$||(L\psi)(\varphi_\epsilon-1)||$, $||\nabla\psi\cdot\nabla(\varphi_\epsilon-1)||$ 
and $||\psi L(\varphi_\epsilon-1)||$.  
Here we just show how to deal with the most difficult term $||\psi L(\varphi_\epsilon-1)||$
where both derivatives fall on the cutoff function, as the computations for the others are similar, but easier. 
We will use the full strength of $d=4$. 
As before, we just need to estimate $||L(\varphi_\epsilon-1)||$. 
We will first calculate this norm for general $1\leq d\leq 4$, then substitute the dimension $d=4$, 
\[||L(\varphi_\epsilon-1)||^2=\int_{\mathbb{R}^d}|L\varphi\big((\frac{|x|}{\epsilon})^\epsilon\big)|^2dx.\]
After computation, we obtain the expression
\begin{equation}
\begin{split}
L\varphi\big((\frac{|x|}{\epsilon})^\epsilon\big)=
(d-2)&\epsilon^{1-\epsilon}|x|^{\epsilon-2}\varphi'\big((\frac{|x|}{\epsilon})^\epsilon\big)\\
&+\epsilon^{2-2\epsilon}|x|^{2\epsilon-2}\varphi''\big((\frac{|x|}{\epsilon})^\epsilon\big)+
\epsilon^{2-\epsilon}|x|^{\epsilon-2}\varphi'\big((\frac{|x|}{\epsilon})^\epsilon\big).
\end{split}
\label{central}
\end{equation}
We will use it in the proofs for all the dimensions $1\leq d\leq 4$. In particular here for $d=4$, it becomes
\[L\varphi\big((\frac{|x|}{\epsilon})^\epsilon\big)=
2\epsilon^{1-\epsilon}|x|^{\epsilon-2}\varphi'\big((\frac{|x|}{\epsilon})^\epsilon\big)+
\epsilon^{2-2\epsilon}|x|^{2\epsilon-2}\varphi''\big((\frac{|x|}{\epsilon})^\epsilon\big)+
\epsilon^{2-\epsilon}|x|^{\epsilon-2}\varphi'\big((\frac{|x|}{\epsilon})^\epsilon\big).\]
Because both $\varphi'$ and $\varphi''$ are bounded, we just need to work with the three terms 
$\epsilon^{1-\epsilon}|x|^{\epsilon-2}$, $\epsilon^{2-2\epsilon}|x|^{2\epsilon-2}$, $\epsilon^{2-\epsilon}|x|^{\epsilon-2}$. 
Because $\varphi'$ and $\varphi''$ vanish outside the ball $B(0,\epsilon)$, we only integrate over the ball $B(0,\epsilon)$. 
The square of these terms are $\epsilon^{2-2\epsilon}|x|^{2\epsilon-4}, \epsilon^{4-4\epsilon}|x|^{4\epsilon-4}$, $\epsilon^{4-2\epsilon}|x|^{2\epsilon-4}$. 
The biggest is $\epsilon^{2-2\epsilon}|x|^{2\epsilon-4}$, so we only need to estimate this one. 
Using polar coordinates,
\[\int_{B(0,\epsilon)}\epsilon^{2-2\epsilon}|x|^{2\epsilon-4}dx=\epsilon^{2-2\epsilon}\int_0^\epsilon r^{2\epsilon-4}r^3dr=
\epsilon^{2-2\epsilon}\int_0^\epsilon r^{2\epsilon-1}dr=\frac{\epsilon}{2}\rightarrow 0\mbox{ as }\epsilon\rightarrow 0.\]
This completes the proof.
\end{proof}

\begin{remark}
Note that this convergence is only logarithmic as a function of the `spread' of $\varphi\big((\frac{|x|}{\epsilon})^\epsilon\big)$, 
as one would expect.
\end{remark}

\subsection{Case d=3}

\noindent We already know the domain 
of $\overline{L}$ lies in $W^{2,2}(\mathbb{R}^3)$. 
By Sobolev Embedding Theorem we know that $W^{2,2}(\mathbb{R}^3)$ can be embedded into $C^{\frac{1}{2}}(\mathbb{R}^3)$, 
therefore it makes sense to talk about the value of one of these functions at a single point, in this case the origin. 
Any function in $D(\overline{L})$ is the limit of a sequence 
of continuous functions that take $0$ at the origin under the $W^{2,2}$-norm, hence it must also 
take $0$ at the origin. Therefore we know, $D(\overline{L})\subseteq\{\psi\in W^{2,2}(\mathbb{R}^3)|\psi(0)=0\}$. 
In fact, they are equal. 

\begin{proposition}
In $\mathbb{R}^3$, we have 
\[D(\overline{L})=\{\psi\in W^{2,2}(\mathbb{R}^3)|\psi(0)=0\}.\]
\end{proposition}
\begin{proof}
We pick any arbitrary $\psi\in C_c^\infty(\mathbb{R}^3)$  with $\psi(0)=0$. We use the same 
sequence as the case $d=4$ to approximate $\psi$. The calculations are similar as well. The difference here is we get
one less power of $r$ when we change into polar coordinates, but this is compensated with the H{\"o}lder condition. 
We show $||\psi\varphi_\epsilon-\psi||, ||(L\psi)(\varphi_\epsilon-1)||, ||\nabla\psi\cdot\nabla(\varphi_\epsilon-1)||\rightarrow 0$ 
as $\epsilon\rightarrow 0$ exactly the same as before.
The only substantially different term is $||\psi L(\varphi_\epsilon-1)||$. By the H{\"o}lder condition, we have 
$|\psi(x)|^2\leq C|x|$ for some constant $C$, it follows that
\begin{equation*}
\begin{split}
||\psi L(\varphi_\epsilon-1)||^2=&\int_{\mathbb{R}^3}|\psi(x) L\varphi\big((\frac{|x|}{\epsilon})^\epsilon\big)|^2dx\\
&=\int_{B(0,\epsilon)}|\psi(x) L\varphi\big((\frac{|x|}{\epsilon})^\epsilon\big)|^2dx
\leq C\int_{B(0,\epsilon)}|x||L\varphi\big((\frac{|x|}{\epsilon})^\epsilon\big)|^2dx.
\end{split}
\end{equation*}
We substitute $d=3$ into (\ref{central}) to obtain expression for $L\varphi\big((\frac{|x|}{\epsilon})^{\epsilon}\big)$,
\[L\varphi\big((\frac{|x|}{\epsilon})^\epsilon\big)=
\epsilon^{1-\epsilon}|x|^{\epsilon-2}\varphi'\big((\frac{|x|}{\epsilon})^\epsilon\big)+
\epsilon^{2-2\epsilon}|x|^{2\epsilon-2}\varphi''\big((\frac{|x|}{\epsilon})^\epsilon\big)+
\epsilon^{2-\epsilon}|x|^{\epsilon-2}\varphi'\big((\frac{|x|}{\epsilon})^\epsilon\big).\]
It is the same as the case $d=4$ except the coefficient for the first term goes down by $1$. 
(Later on when we discuss the case $d=2$, this term will disappear.) 
The three terms we need to work with are still $\epsilon^{1-\epsilon}|x|^{\epsilon-2}$, $\epsilon^{2-2\epsilon}|x|^{2\epsilon-2}$ and $\epsilon^{2-\epsilon}|x|^{\epsilon-2}$. 
Square these terms and multiply with $|x|$, we get 
$\epsilon^{2-2\epsilon}|x|^{2\epsilon-3}$, $\epsilon^{4-4\epsilon}|x|^{4\epsilon-3}$ and $\epsilon^{4-2\epsilon}|x|^{2\epsilon-3}$. 
Parallel to before, we integrate them over the ball $B(0,\epsilon)$, and we check the worst term to complete the proof,
\[\int_{B(0,\epsilon)}\epsilon^{2-2\epsilon}|x|^{2\epsilon-3}dx=\epsilon^{2-2\epsilon}\int_0^\epsilon r^{2\epsilon-3}r^2dr=
\epsilon^{2-2\epsilon}\int_0^\epsilon r^{2\epsilon-1}dr=\frac{\epsilon}{2}\rightarrow 0\mbox{ as }\epsilon\rightarrow 0.\]
\end{proof}

\subsection{Case d=2}

\noindent Here by Sobolev Embedding Theorem, we have $W^{2,2}(\mathbb{R}^2)\subseteq C^{\gamma}(\mathbb{R}^2)$ for any $\gamma<1$. 
Unfortunately we don't get $1$, so the increase in the H{\"o}lder exponent is not enough to compensate the loss in one power when changing into polar coordinates. 
Instead here for any $\psi\in D(\overline{L})$, we have $|\psi(x)|=|\psi(x)-\psi(0)|\leq C|x|(\ln\frac{1}{|x|})^{\frac{1}{2}}$ for some constant $C>0$, 
and for small $|x|$, see \cite[Sec. 4.1]{MT}.

\begin{proposition}
For $\mathbb{R}^2$, we have
\[D(\overline{L})=\{\psi\in W^{2,2}(\mathbb{R}^2)|\psi(0)=0\}.\]
\end{proposition}

\begin{proof}
The terms $||\psi\varphi_\epsilon-\psi||, ||(L\psi)(\varphi_\epsilon -1)||$ and $||\nabla\psi\cdot\nabla(\varphi_\epsilon-1)||$ are estimated like before. 
Now we work on the term $||\psi L(\varphi_\epsilon-1)||$. 
As stated above, here we have 
\[|\psi(x)|\leq C|x|(\ln\frac{1}{|x|})^{\frac{1}{2}},\]
for some constant $C>0$. Then 
\[||\psi L(\varphi_\epsilon-1)||^2=\int_{\mathbb{R}^3}|\psi(x) L\varphi\big((\frac{|x|}{\epsilon})^\epsilon\big)|^2dx
\leq C^2\int_{B(0,\epsilon)}|x|^2\ln\frac{1}{|x|}|L\varphi\big((\frac{|x|}{\epsilon})^\epsilon\big)|^2dx.\]
Again as before, we substitute $d=2$ into (\ref{central}) to obtain the expression for $L\varphi\big((\frac{|x|}{\epsilon})^{\epsilon}\big)$,
\[L\varphi\big((\frac{|x|}{\epsilon})^\epsilon\big)=
\epsilon^{2-2\epsilon}|x|^{2\epsilon-2}\varphi''\big((\frac{|x|}{\epsilon})^\epsilon\big)+
\epsilon^{2-\epsilon}|x|^{\epsilon-2}\varphi'\big((\frac{|x|}{\epsilon})^\epsilon\big).\]
Note that in this dimension the first term in ($\ref{central}$), $\epsilon^{1-\epsilon}|x|^{\epsilon-2}\varphi'\big((\frac{|x|}{\epsilon})^\epsilon\big)$, is gone, 
which is good news as it was previously the biggest term. 
Because $\varphi'$ and $\varphi''$ are bounded, we estimate $\epsilon^{2-2\epsilon}|x|^{2\epsilon-2}$ and $\epsilon^{2-\epsilon}|x|^{\epsilon-2}$. 
We square them and then multiply with $|x|^2\ln\frac{1}{|x|}$ to obtain $\epsilon^{4-4\epsilon}|x|^{4\epsilon-2}\ln\frac{1}{|x|}$ and 
$\epsilon^{4-2\epsilon}|x|^{2\epsilon-2}\ln\frac{1}{|x|}$. The second term is bigger, 
\begin{equation*}
\begin{split}
\int_{B(0,\epsilon)}\epsilon^{4-2\epsilon}|x|^{2\epsilon-2}\ln\frac{1}{|x|}dx
=-\epsilon^{4-2\epsilon}\int_0^\epsilon r^{2\epsilon-1}\ln rdr
=-\frac{\epsilon^3\ln\epsilon}{2}+\frac{\epsilon^2}{4}.
\end{split}
\end{equation*}

\noindent This expression goes to $0$ as $\epsilon$ approaches $0$. This completes the proof.
\end{proof}

\subsection{Case d=1}

In this case the power of $r$ we obtained from changing into polar coordinates no longer exists, instead we have H{\"o}lder condition 
for the derivatives as compensation. This means the domain is different from the above cases as we must impose condition on the derivatives.  
Here, by Sobolev Embedding Theorem, not only $W^{2,2}(\mathbb{R})\subseteq C^{\frac{1}{2}}(\mathbb{R})$ we also have $W^{1,2}(\mathbb{R})\subseteq C^{\frac{1}{2}}(\mathbb{R})$, which means it makes sense to talk about the derivative of a function in the domain at a single point as well. 
We will skip the proof of the following proposition as it is similar to the other cases.

\begin{proposition}
For $\mathbb{R}$, we have
\[D(\overline{L})=\{\psi\in W^{2,2}(\mathbb{R})|\psi(0)=0, \psi'(0)=0\}.\] 
\end{proposition}

\section{Self-adjoint extensions of $L$}\label{opdef}

\noindent After finding the closures we are now in the position to calculate the self-adjoint extensions. 

\subsection{Case d$\geq$4}
In this case $\overline{L}$ is already self-adjoint. We know that by calculating the deficiency indices of $\overline{L}$. 
Suppose that $u\in\mathscr{K}_+$, it means, 
\[\big(u,(\overline{L} +i)f\big)=0\mbox{ for all }f\in W^{2,2}(\mathbb{R}^d).\]
In Fourier space, we have $\big(\hat{u},(|\xi|^2+i\big)\hat{f})=0$ for all $f\in W^{2,2}(\mathbb{R}^d)$. 
Then $\big((|\xi|^2-i)\hat{u},\hat{f}\big)=0$ for all $f\in W^{2,2}(\mathbb{R}^d)$. 
Since $\mathscr{F}(W^{2,2}(\mathbb{R}^d))$ is dense in $L^2(\mathbb{R}^d)$, we have $(|\xi|^2-i)\hat{u}=0$. 
It follows that $u=0$, ie $\mathscr{K}_+=\{0\}$. Similarly, $\mathscr{K}_-=\{0\}$.

\subsection{Case d=2, 3}

Here for $u\in\mathscr{K}_+$, we have 
\[\big(u,(\overline{L} +i)f\big)=0\mbox{ for all }f\in W^{2,2}(\mathbb{R}^d)\mbox{ such that }f(0)=0.\]
Again we work in Fourier space, then the condition becomes 
\[\big(\hat{u},(|\xi|^2 +i)\hat{f}\big)=0\mbox{ for all }\hat{f}\in (1+|\xi|^2)^{-1}L^2(\mathbb{R}^d)\mbox{ such that }\int_{\mathbb{R}^d}\hat{f}(\xi)d\xi=0.\]
Let $\varphi$ be a fixed function in $W^{2,2}(\mathbb{R}^d)$ such that $\int_{\mathbb{R}^d}\hat{\varphi}(\xi)d\xi=1$. 
Then for any $f\in W^{2,2}(\mathbb{R}^d)$, the function $\hat{f}-\int_{\mathbb{R}^d}\hat{f}(x)dx\hat{\varphi}$ is in $(1+|\xi|^2)^{-1}L^2(\mathbb{R}^d)$, 
and its integral over $\mathbb{R}^d$ is $0$. Hence we have 
\[\Big(\hat{u}, \big(|\xi|^2+i\big)\big(\hat{f}-\int_{\mathbb{R}^d}\hat{f}(x)dx\hat{\varphi}\big)\Big)=0.\]
That means 
\[\Big((|\xi|^2-i)\hat{u}-\big((|\xi|^2-i)\hat{u},\hat{\varphi}\big),\hat{f}\Big)=0.\]
Since $W^{2,2}(\mathbb{R}^d)$ is dense in $L^2(\mathbb{R}^d)$, we conclude that $(|\xi|^2-i)\hat{u}=c$ where $c=((|\xi|^2-i)\hat{u},\hat{\varphi})$ is a constant. 
Thus $\hat{u}=\frac{c}{|\xi|^2-i}$, so we know
$\mathscr{K}_+$ is the one dimensional complex space spanned by $\mathcal{F}^{-1}(\frac{1}{|\xi|^2-i})$. 
Similarly, $\mathscr{K}_-$ is the one dimensional complex space spanned by $\mathcal{F}^{-1}(\frac{1}{|\xi|^2+i})$.
Since both $\mathscr{K}_+$ and $\mathscr{K}_-$ are one dimensional, and $\frac{1}{|\xi|^2-i}$, $\frac{1}{|\xi|^2+i}$ have the same norms, 
the unitary maps from $\mathscr{K}_+$ onto $\mathscr{K}_-$ can be parametrized by the unit circle in the complex plane, such that for each $\theta\in[-\pi,\pi)$, 
$\mathcal{F}^{-1}(\frac{1}{|\xi|^2-i})$ is mapped to $e^{i\theta}\mathcal{F}^{-1}(\frac{1}{|\xi|^2+i})$.
Therefore we have, 

\begin{proposition}\label{dom}
The self-adjoint extensions of $L$ in $\mathbb{R}^d$, $d=2,3$, can be parametrized by a circle $\theta\in[-\pi, \pi)$, with
\[D(L_\theta)=
\{\varphi+\beta\mathcal{F}^{-1}(\frac{1}{|\xi|^2-i}) +e^{i\theta}\beta\mathcal{F}^{-1}(\frac{1}{|\xi|^2+i}):\varphi\in W^{2,2}(\mathbb{R}^d), \varphi(0)
=0\mbox{  and }\beta\in\mathbb{C}\},\]
and 
\[L_\theta(\varphi+\beta\mathcal{F}^{-1}(\frac{1}{|\xi|^2-i})+e^{i\theta}\beta\mathcal{F}^{-1}(\frac{1}{|\xi|^2+i}))
=\overline{L}\varphi+i\beta\mathcal{F}^{-1}(\frac{1}{|\xi|^2-i})-ie^{i\theta}\beta\mathcal{F}^{-1}(\frac{1}{|\xi|^2+i}).\]
\end{proposition}

\begin{remark}
Notice that $\frac{1}{|\xi|^2- i}$, $\frac{1}{|\xi|^2+i}$ are in $L^2(\mathbb{R}^d)$ iff $d<4$, 
hence they cannot be in the deficiency subspaces of $\overline{L}$ for any $d\geq 4$.
\end{remark}

\noindent Among all the self-adjoint extensions of $L$, $L_{-\pi}$ is the special one. 
We denote it by $\Delta$, and call it the Laplacian. It can be easily shown that
\[D(\Delta)=W^{2,2}(\mathbb{R}^d),\hspace{5mm}d=2,3.\]
It is the only self-adjoint extension whose domain 
is contained in the form domain of the closure of
the quadratic form $q$ associated with $L$:
\[q(\varphi,\psi)=\int_{\mathbb{R}^d} \nabla\varphi\cdot\overline{\nabla\psi}dx.\]
Note that $D(\Delta)=W^{2,2}(\mathbb{R}^d)$ is contained in $Q(\hat{q})=W^{1,2}(\mathbb{R}^d)$, 
 it follows from \cite[Sec. X.3]{RS} that $\Delta$  is the Friedrichs extension.

\vspace{5mm}

\noindent For $d=3$, let's look at another way to parametrize these self-adjoint extensions by
considering the Taylor expansions at the origin of the functions in the domains of these extensions.  
The parameter is the ratio between the constant term and the coefficient of the $\frac{1}{|x|}$ term in the expansions. 
For that, we use the following well known result, 
\begin{equation}\label{expansion}
\mathscr{F}^{-1}(\frac{1}{|\xi|^2+\lambda^2})=\frac{1}{4\pi}\frac{e^{-\lambda|x|}}{|x|},
\end{equation}
for $\lambda\in\mathbb{C}\backslash i\mathbb{R}$. From this we obtain these two expansions at the origin, 
\[\mathscr{F}^{-1}(\frac{1}{|\xi|^2-i})=\frac{1}{4\pi}\frac{e^{-\cis(-\frac{\pi}{4})|x|}}{|x|}=\frac{1}{4\pi|x|}-\frac{1}{4\pi}e^{-\frac{\pi}{4}i}+O(|x|),\]
\[\mathscr{F}^{-1}(\frac{1}{|\xi|^2+i})=\frac{1}{4\pi}\frac{e^{-\cis(\frac{\pi}{4})|x|}}{|x|}=\frac{1}{4\pi|x|}-\frac{1}{4\pi}e^{\frac{\pi}{4}i}+O(|x|).\]

\noindent We will use the above expansions in the proof of the following proposition,

\begin{proposition}\label{sec}
The self-adjoint extensions of $L$ in $\mathbb{R}^3$ can be parametrized by 
$\mu\in\mathbb{R}\cup\{\infty\}$ with 
\[D(L^\mu)=\{\varphi=\chi\beta(\frac{1}{|x|}+\mu)+\varphi': \varphi'\in W^{2,2}(\mathbb{R}^3), \varphi'(0)=0, \beta\in\mathbb{C}\},\]
where $\chi$ is a function in $C_c^\infty(\mathbb{R}^3)$ and 
$\chi\equiv 1$ near $0$, and 
\[D(L^\infty)=W^{2,2}(\mathbb{R}^3).\]
Moreover, the relationship with the previous parametrisation is
 \[L^\mu=L_\theta\mbox{ iff }\mu(\theta)=\frac{\sqrt{2}}{2}\big(\tan(\frac{\theta}{2})-1\big),\hspace{5mm}\theta\in(-\pi,\pi),\]
 and 
 \[L^\infty=L_{-\pi}=\Delta.\]
 \end{proposition}
 
\begin{remark}
Note that the case $\theta=-\pi$ corresponds to functions with only constant term but no $\frac{1}{|x|}$ term in the expansion at the origin.  
The case $\theta=\frac{\pi}{2}$ corresponds to functions 
with only $\frac{1}{|x|}$ term but no constant term in the expansion at the origin. 
\end{remark}
 
\begin{proof}
We focus on the first two terms in the expansion. 
From Proposition \ref{dom} and equation (\ref{expansion}) we know 
they are a multiple of 
\[(1+e^{i\theta})\frac{1}{|x|}-(e^{-\frac{\pi}{4}i}+e^{(\theta+\frac{\pi}{4})i}).\]
For $\theta=-\pi$, the $\frac{1}{|x|}$ term disappears, and we have a non-zero constant $-\sqrt{2}$. 
This means we can get any value at the origin, which is consistent with what we already know. 
In the new parametrisation, we label this operator $L^\infty$. 

\vspace{3mm}

\noindent Then we consider $\theta\ne -\pi$, here the first two terms in the expansion are a multiple of 
\[\frac{1}{|x|}-\frac{e^{-\frac{\pi}{4}i}+e^{(\theta+\frac{\pi}{4})i}}{1+e^{i\theta}}.\]
Denote 
\begin{equation}\label{mu1}
\mu(\theta)=-\frac{e^{-\frac{\pi}{4}i}+e^{(\theta+\frac{\pi}{4})i}}{1+e^{i\theta}}=-e^{\frac{\pi}{4}i}(\frac{e^{i\theta}-i}{e^{i\theta}+1}),
\end{equation}
and in the new parametrisation this operator is denoted by $L^\mu$. 
After some calculation, we know that
\begin{equation}\label{mu2}
\mu(\theta)=\frac{\sqrt{2}}{2}\big(\tan(\frac{\theta}{2})-1\big),\hspace{5mm}\theta\in(-\pi,\pi).
\end{equation}
From this expression, we can see $\mu(\theta)$ ranges across the real line. When $\theta$ increases from $-\pi$ to $\pi$, 
$\mu(\theta)$ moves rightwards along the real line from $-\infty$ to $\infty$, 
and it changes sign when $\theta=\frac{\pi}{2}$. 
\end{proof}

\begin{remark}
The $\mu$-parametrization corresponds to the usual way self-adjoint extensions of the Laplacian on a cone are defined, 
that is, in terms of the expansions of harmonic functions at the cone point, see \cite{EM}.
\end{remark}

\section{Resolvent Kernel}\label{specfns}

\subsection{Resolvent Kernel}

\noindent From now on we focus on $\mathbb{R}^3$. 
In this section we calculate the resonances and eigenvalues of the self-adjoint extensions of $L$. 
For that we need to determine the resolvent kernel of these various self-adjoint extensions.
We start with $\Delta$. 
Since the spectrum of a self-adjoint operator is a subset of $\mathbb{R}$,
for any $\lambda\notin\mathbb{R}$, $\lambda^2$ is in the resolvent set, ie
$(\Delta-\lambda^2)^{-1}$ exists. Let's determine the kernel of $(\Delta-\lambda^2)^{-1}$ for $\lambda\in\mathbb{C}$ with $Im(\lambda)>0$. 
Suppose $(\Delta-\lambda^2)u=f$. Taking Fourier transform, we have $(|\xi|^2-\lambda^2)\hat{u}=\hat{f}$.
Hence, $\hat{u}=\frac{\hat{f}}{|\xi|-\lambda^2}$, so
\[(\Delta-\lambda^2)^{-1}f=u=\mathscr{F}^{-1}(\frac{\hat{f}}{|\xi|^2-\lambda^2})=f*\mathscr{F}^{-1}(\frac{1}{|\xi^2|-\lambda^2})=\frac{1}{4\pi}f*\frac{e^{i\lambda|x|}}{|x|}.\]
Hence we know the kernel of the operator $(\Delta-\lambda^2)^{-1}$ is 
\begin{equation}\label{free}
K_{free}(\lambda, x, y)=\frac{e^{i\lambda|x-y|}}{4\pi|x-y|}.
\end{equation}
Now we determine the kernel of $(L^\mu-\lambda^2)^{-1}$, $\mu\in(-\infty,\infty)$. 
Besides the term $K_{free}(\lambda, x, y)$, the kernel here has another term $K_{extra}(\mu, \lambda, x, y)$. 
We guess that
\[K_{extra}(\mu, \lambda, x, y)=b(\mu,\lambda)\frac{e^{i\lambda(|x|+|y|)}}{|x||y|},\]
for some $b(\mu, \lambda)$. 
With this guess, we first determine what $b(\mu, \lambda)$ must be, 
then verify it is indeed the kernel what we are after.
Denote
\[K(\mu, \lambda, x, y)=K_{free}(\lambda, x, y)+K_{extra}(\mu, \lambda, x, y)=\frac{e^{i\lambda|x-y|}}{4\pi|x-y|}+b(\mu,\lambda)\frac{e^{i\lambda(|x|+|y|)}}{|x||y|}.\]
If $K(\mu, \lambda, x, y)$ is indeed the kernel of $(L^\mu-\lambda^2)^{-1}$, it must lie in the domain of $L^\mu$ when $y$ is fixed. 
So we fix $y$ and consider the expansion of the function at $x=0$,
\[\frac{e^{i\lambda|y|}}{|y|}\Big(b(\mu,\lambda)\frac{1}{|x|}+\big(\frac{1}{4\pi}+i\lambda b(\mu,\lambda)\big)\Big).\]
When $\mu\ne\infty$, $b(\mu,\lambda)\ne 0$, and being in the domain of $L^\mu$ means
\[\frac{\frac{1}{4\pi}+i\lambda b(\mu,\lambda)}{b(\mu,\lambda)}=\mu.\]
We solve for $b(\mu,\lambda)$, 
\[b(\mu,\lambda)=\frac{i}{4\pi(\lambda+i\mu)}.\]
Therefore 
\[K_{extra}(\mu,\lambda, x, y)=\frac{ie^{i\lambda(|x|+|y|)}}{4\pi |x||y|(\lambda +i\mu)},\]
and the kernel of $(L^\mu-\lambda^2)^{-1}$ is
\begin{equation}\label{kernel}
K(\mu, \lambda, x, y)=\frac{e^{i\lambda|x-y|}}{4\pi|x-y|}+\frac{ie^{i\lambda(|x|+|y|)}}{4\pi|x||y|(\lambda+i\mu)}.
\end{equation}
We are left to verify this is the correct kernel.
Indeed, it maps $L^2(\mathbb{R}^3)$ into $D(L^\mu)$, and away from the origin, we have
\[(L^\mu-\lambda^2)K_{\mu,\lambda}f=(\Delta-\lambda^2)K_{\mu,\lambda}f=f,\]
where $K_{\mu,\lambda}$ denotes the operator corresponds to the kernel $K(\mu,\lambda, x, y)$. 
Since the origin has measure $0$, it means in $L^2(\mathbb{R}^3)$ we have
\[(L^\mu-\lambda^2)K_{\mu, \lambda}f=f,\]
which shows we have found the correct kernel. 

\subsection{Resonance}

\noindent Remember we determine the above kernel of $\lambda\in\mathbb{C}$ with $Im(\lambda)>0$. 
In this region, we have exponential decay, so the operator maps $L^2$-functions to $L^2$-functions. 
The kernel $K(\mu, \lambda, x, y)$, as a function of $\lambda$, 
clearly has a meromoprhic continuation 
to the whole complex plane $\mathbb{C}$. The continuation is defined by the same expression, so for convenience, 
we use the same name $K(\mu, \lambda, x, y)$ to denote the continuation. 
We see from (\ref{kernel}) that $K(\mu, \lambda, x, y)$ has a single pole at $\lambda=-i\mu$.
For $\mu<0$, this pole is in the physical half of the plane, and as we will see in the next subsection, its square 
is an eigenvalue of the operator $L^\mu$. 
For $\mu\geq 0$, the pole is in the non-physical half of the plane, but it still has physical significance as
shown in the next section, and in this case the pole is called a resonance. 

\begin{remark}
Note that the Laplacian $\Delta$ is the only self-adjoint extension of $L$ without either an eigenvalue or a resonance.
\end{remark}

\subsection{Eigenvalue}
\noindent 
We start with an arbitrary $\lambda$, and try to find an eigenfunction in $D(L^*)$, then determine whether it lies in the domain of any self-adjoint extension of $L$. 
Suppose $\varphi$ is an eigenfunction in $D(L^*)$ with eigenvalue $\lambda$, that means it lies in $\mbox{Ker}(L^*-\lambda)=\mbox{Ran}(\overline{L}-\lambda)^\bot$, hence 
\[\big(\varphi,(\overline{L}-\lambda)\psi\big)=0\hspace{5mm}\mbox{ for all }\psi\in D(\overline{L}).\]
We can solve the above equation similar as before. 
It has a non-trivial solution only when $\lambda$ is negative, so we know none of 
the self-adjoint extensions has a non-negative eigenvalue. 
For $\lambda<0$, the eigenspace corresponding to $\lambda$ is spanned by 
\[\mathscr{F}^{-1}(\frac{1}{|\xi|^2-\lambda})=\frac{1}{4\pi}\frac{e^{-\sqrt{-\lambda}|x|}}{|x|}=\frac{1}{4\pi|x|}-\frac{\sqrt{-\lambda}}{4\pi}+O(|x|).\]
We use equation (\ref{expansion}) to obtain the expansion at the origin in the above equation. 
By Proposition \ref{sec} we know this function is in $D(L^{-\sqrt{-\lambda}})$. 

\vspace{2mm}

\noindent From above discussion we know that for $\mu\in (-\infty, 0)$, the extension $L^\mu$ has an eigenvalue $-\mu^2$, and 
the eigenspace is spanned by the function below,
\[v_\mu(x)=\frac{e^{\mu|x|}}{|x|}.\]
Since $\mu$ is negative,  
$v_\mu$ is in $L^2(\mathbb{R}^3)$. We compute that $||v_{\mu}||_2^2=-\frac{2\pi}{\mu}$, 
so a normalised eigenfunction corresponding to the eigenvalue $-\mu^2$ is 
\[\sqrt{\frac{-\mu}{2\pi}}\frac{e^{\mu |x|}}{|x|}.\]

\vspace{3mm}
\noindent When $\mu\in [0,\infty)$, that is when $\theta\in [\frac{\pi}{2},\pi)$, 
the expression of the eigenfunction $v_\mu$ in the above case is no longer in $L^2(\mathbb{R}^3)$. 
As mentioned earlier, in this case the pole $a(\mu)=-i\mu$ is a resonance.
To summarise, the spectra of the self-adjoint extensions are
\begin{equation*}
\begin{split}
&\sigma(L_\mu)=\{-\mu^2\}\cup[0,\infty),\hspace{3mm}\mu\in(-\infty,0),\\
&\sigma(L_\mu)=[0,\infty),\hspace{3mm}\mu\in[0,\infty].
\end{split}
\end{equation*}

\section{A wave equation involving $L^\mu$}\label{res-kernels}

\subsection{The wave equation}

\noindent For $\mu\in [0,\infty)$, the resonance $-i\mu$ doesn't result in an eigenvalue, and it is in the
non-physical half of the complex plane as it means exponential growth of the kernel $K(\mu, \lambda, x, y)$.
But this resonance still has physical significance, and as we will see, it appears in the 
wave kernel involving $L^\mu$, see \cite{LP}.
The wave equation we are going to consider is
\[\begin{cases}
\partial^2_t u+ L^\mu u=0\\
u|_{t=0}=f\\
\partial_tu|_{t=0}=g,
\end{cases}\]
where $\mu\in\mathbb{R}\cup\{\infty\}$, $f, g\in C_c^\infty(\mathbb{R}^3\backslash\{0\})$. 

\noindent We know that the solution for the system 
\[\begin{cases}
\partial^2_t u+ a^2 u=0\\
u|_{t=0}=f\\
\partial_tu|_{t=0}=g,
\end{cases}\]
where $a\in\mathbb{R}$, is
\[u(t)=\cos(at)f+\frac{\sin(at)}{a}g.\]
So by functional calculus, the solution for the system we are interested in is 
\[u(t)=\cos(t\sqrt{L^\mu})f+\frac{\sin(t\sqrt{L^\mu})}{\sqrt{L^\mu}}g.\]
Here if $f\in D(L^\mu)$ and $g\in D(\sqrt{L^\mu})$, we have a strong solution, ie $u(t)\in D(L^\mu)$ for each $t$, 
and $u$ is continuous as a function of $t$ with values in $D(L^\mu)$. 

\vspace{2mm}

\noindent We proceed to calculate the kernel of $\frac{\sin(t\sqrt{L^\mu})}{\sqrt{L^\mu}}$, then 
the kernel of $\cos(t\sqrt{L^\mu})$ is given by its time derivative.

\begin{proposition} 
For any $\mu\in\mathbb{R}$, the kernel of $\frac{\sin(t\sqrt{L^\mu})}{\sqrt{L^\mu}}$ for $t\geq 0$ is
\begin{equation}\label{cker}
\frac{1}{4\pi}
\big(
\delta(t^2-|x-y|^2)
+\frac{1}{|x||y|}H(t-|x|-|y|)e^{\mu(|x|+|y|-t)}\big),
\end{equation}
where $H$ is the Heaviside function.
\end{proposition}

\begin{remark}
We know that away from the origin, $L^\mu$ is the  same as $\Delta$. 
Also, due to finite propagation speed, the minimum time required to travel from 
$x$ to $y$ through the origin is $|x|+|y|$. 
Therefore for $t<|x|+|y|$, 
we would expect that 
$\frac{\sin(t\sqrt{L^\mu})}{\sqrt{L^\mu}}$ has the same kernel as $\frac{\sin(t\sqrt{\Delta})}{\sqrt{\Delta}}$.
Our kernel (\ref{cker}) satifies this, and is a check on its correctness.
\end{remark}

\begin{remark}
The second term in (\ref{cker}) can be interpreted as a 
``diffracted'' wave from the origin thought of as a cone point. 
The strength of the singularity is $1$ order 
weaker than the incident singularity, as is the case for a diffracted wave, 
see \cite{CT} and \cite{MW}.
\end{remark}

\begin{remark}
When $\mu<0$, the second term of (\ref{cker}) is exponentially growing in time as $t\rightarrow\infty$, 
but exponentially decaying in space as $|x|, |y|\rightarrow\infty$.
This is due to the negative eigenvalue. 
On the other hand, when $\mu>0$, this term is exponentially decaying in time as $t\rightarrow\infty$, 
but exponentially growing in space as $|x|, |y|\rightarrow\infty$. 
It corresponds to a term in the ``resonance expansion'' for solutions to the wave equation on 
a compact set, see \cite{LP} and \cite{TZ}.
\end{remark}

\begin{remark}
See \cite{LH} for a different approach to obtain the kernel by solving an auxiliary problem.
\end{remark}

\begin{proof}
\noindent We have, 
\[\frac{\sin(t\sqrt{L^\mu})}{\sqrt{L^\mu}}=\int_{-\infty}^\infty\frac{\sin(t\sqrt{\sigma})}{\sqrt{\sigma}}dP_\sigma
=\lim_{R\rightarrow\infty}\int_{-\infty}^\infty\frac{\sin(t\sqrt{\sigma})}{\sqrt{\sigma}}\varphi(\frac{\sigma}{R})dP_\sigma,\]
where $\varphi :\mathbb{R}\rightarrow [0,1]$ is a smooth function such that $\varphi=1$ on $B_1(0)$ and $\varphi=0$ outside $B_2(0)$, and 
the limit converges under the strong operator topology.
Then depending on whether $L^\mu$ has an eigenvalue, we have two possibilities. First for $\mu\in [0,\infty)$, we have
\[\frac{\sin(t\sqrt{L^\mu})}{\sqrt{L^\mu}}=\lim_{R\rightarrow\infty}\int_{-\infty}^\infty\frac{\sin(t\sqrt{\sigma})}{\sqrt{\sigma}}\varphi(\frac{\sigma}{R})dP_\sigma
=\lim_{R\rightarrow\infty}\int_0^\infty\frac{\sin(t\sqrt{\sigma})}{\sqrt{\sigma}}\varphi(\frac{\sigma}{R})dP_\sigma .\]
While for $\mu(-\infty, 0)$ the eigenvalue contributes an extra term, 
\[\frac{\sin(t\sqrt{L^\mu})}{\sqrt{L^\mu}}=\lim_{R\rightarrow\infty}\int_{-\infty}^\infty\frac{\sin(t\sqrt{\sigma})}{\sqrt{\sigma}}\varphi(\frac{\sigma}{R})dP_\sigma
=\lim_{R\rightarrow\infty}\int_0^\infty\frac{\sin(t\sqrt{\sigma})}{\sqrt{\sigma}}\varphi(\frac{\sigma}{R})dP_\sigma 
+\frac{\sin(i\mu t)}{i\mu}P_{-\mu^2},\]
where $P_{-\mu^2}$ is the orthogonal projection onto the eigenspace of $-\mu^2$.

\vspace{2mm} 

\noindent In either case we need to calculate the term $\lim_{R\rightarrow\infty}\int_0^\infty\frac{\sin(t\sqrt{\sigma})}{\sqrt{\sigma}}\varphi(\frac{\sigma}{R})dP_\sigma$. 
For each $R>0$, we apply integration by parts twice to evaluate the integral, and also by using Stone's formula, we obtain, 
\begin{equation*}
\begin{split}
&\int_0^\infty\frac{\sin(t\sqrt{\sigma})}{\sqrt{\sigma}}\varphi(\frac{\sigma}{R})dP_\sigma\\
=&\frac{1}{2\pi i}\int_0^\infty\frac{\sin(t\sqrt{\sigma})}{\sqrt{\sigma}}\varphi(\frac{\sigma}{R})
\lim_{\epsilon\downarrow 0}\big((L^\mu-\sigma-i\epsilon)^{-1}-(L^\mu-\sigma+i\epsilon)^{-1}\big)d\sigma.\\
\end{split}
\end{equation*}

\noindent Remember the kernel of $(L^\mu-\lambda^2)^{-1}$, $\mu\in\mathbb{R}\cup\{\infty\}$ is dented by $K(\mu, \lambda, x, y)$, 
hence for any function $f\in D(\sqrt{L^\mu})$ we have
\begin{equation}\label{sine}
\begin{split}
&\int_0^\infty\frac{\sin(t\sqrt{\sigma})}{\sqrt{\sigma}}\varphi(\frac{\sigma}{R})dP_\sigma\big(f(x)\big)\\
=&\frac{1}{2\pi i}\int_0^\infty\frac{\sin(t\sqrt{\sigma})}{\sqrt{\sigma}}\varphi(\frac{\sigma}{R})\int_{-\infty}^\infty
\big(K(\mu, \sqrt{\sigma}, x, y)-K(\mu, -\sqrt{\sigma}, x, y)\big)f(y)dyd\sigma\\
=&\frac{1}{\pi i}\int_{-\infty}^\infty\int_{-\infty}^\infty\sin(t\lambda)K(\mu,\lambda, x, y)\varphi(\frac{\lambda^2}{R})f(y)d\lambda dy
\hspace{5mm}(\mbox{substitute }\lambda=\sqrt{\sigma}).\\
\end{split}
\end{equation}

\noindent From before, we know that
\[K(\mu, \lambda, x, y)=K_{free}(\lambda, x, y)+K_{extra}(\mu, \lambda, x, y)
=\frac{e^{i\lambda|x-y|}}{4\pi|x-y|}+\frac{ie^{i\lambda(|x|+|y|)}}{4\pi|x||y|(\lambda+i\mu)},\]

\noindent The free resolvent kernel $K_{free}(\lambda,x,y)$
substituted to the last line of (\ref{sine}) gives the free wave kernel $\frac{1}{4\pi}\delta(t^2-|x-y|^2)$, 
so from now on we concentrate on the term contributed by $K_{extra}(\mu, \lambda, x, y)$,
\begin{equation}\label{concentrate}
\frac{1}{\pi i}\lim_{R\rightarrow\infty}\int_{-\infty}^\infty\sin(t\lambda)K_{extra}(\mu,\lambda,x,y)\varphi(\frac{\lambda^2}{R})d\lambda.
\end{equation}

\noindent We continue the computation in cases depending on the sign of $\mu$.

\vspace{2mm}

\noindent {\bf Case 1: when $\mu=0$.}

\noindent In this case equation (\ref{concentrate}) becomes
\begin{equation}\label{zerocase}
\begin{split}
&\frac{1}{4\pi^2|x||y|}\lim_{R\rightarrow\infty}\int_{-\infty}^\infty\frac{1}{\lambda}\sin(t\lambda)e^{i\lambda(|x|+|y|)}\varphi(\frac{\lambda^2}{R})d\lambda\\
=&\frac{1}{4\pi^2|x||y|}\lim_{R\rightarrow\infty}\int_{-\infty}^\infty\frac{1}{\lambda}\sin(t\lambda)\cos\big(\lambda(|x|+|y|)\big)\varphi(\frac{\lambda^2}{R})d\lambda\\
=&\frac{1}{8\pi^2|x||y|}\lim_{R\rightarrow\infty}
\Big(\int_{-\infty}^\infty\frac{1}{\lambda}\sin\big(\lambda(t+|x|+|y|)\big)\varphi(\frac{\lambda^2}{R})d\lambda\\
&\hspace{50mm}+\int_{-\infty}^\infty\frac{1}{\lambda}\sin\big(\lambda(t-|x|-|y|)\big)\varphi(\frac{\lambda^2}{R})d\lambda\Big),\\
\end{split}
\end{equation}

\noindent We then make a substitution, and split into three cases:

\begin{enumerate}
\item[\rm (i)] When $t-|x|-|y|>0$, (\ref{zerocase}) becomes
\begin{equation}\label{fc}
\begin{split}
&\frac{1}{8\pi^2|x||y|}\Bigg(\lim_{R\rightarrow\infty}\int_{-\infty}^\infty\frac{\sin\lambda}{\lambda}\varphi\Big(\frac{\lambda^2}{R(t+|x|+|y|)^2}\Big)d\lambda\\
&\hspace{50mm}+\lim_{R\rightarrow\infty}\int_{-\infty}^\infty\frac{\sin\lambda}{\lambda}\varphi\Big(\frac{\lambda^2}{R(t-|x|-|y|)^2}\Big)d\lambda\Bigg),\\
\end{split}
\end{equation}
note that we have two Dirichlet integrals, each of which equals $\pi$, so the above expression equals
\[\frac{1}{4\pi|x||y|.}\]

\item[\rm (ii)] When $t-|x|-|y|<0$, after the substitution we have the same 
expression as (\ref{fc}) except the sign of the second integral is negative. 
Hence the two integrals cancel each other, therefore in this case (\ref{zerocase}) equals $0$.

\item[\rm (iii)] When $t-|x|-|y|=0$, the second integral 
in the last line of equation (\ref{zerocase}) is $0$, 
hence we only get the first integral in expression (\ref{fc}), therefore
here (\ref{zerocase}) equals $\frac{1}{8\pi|x||y|}$.
\end{enumerate}

\noindent Combine all three cases, the kernel of $\frac{\sin(t\sqrt{L^0})}{\sqrt{L^0}}$ contributed by $K_{extra}$ for $t\geq 0$ is 
\[\frac{1}{4\pi |x||y|}H(t-|x|-|y|).\]

\noindent {\bf Case 2: when $\mu\ne 0$.}

\noindent Since $\sin(t\lambda)=\frac{e^{it\lambda}-e^{-it\lambda}}{2i}$, we have to deal with the following two terms
\begin{equation}\label{first}
\frac{-i}{8\pi^2 |x||y|}\lim_{R\rightarrow\infty}\int_{Im(\lambda)=0}\frac{e^{i\lambda(|x|+|y|+t)}}{\lambda+i\mu}\varphi(\frac{\lambda^2}{R})d\lambda,
\end{equation}
and
\begin{equation}\label{second}
\frac{i}{8\pi^2 |x||y|}\lim_{R\rightarrow\infty}\int_{Im(\lambda)=0}\frac{e^{i\lambda(|x|+|y|-t)}}{\lambda+i\mu}\varphi(\frac{\lambda^2}{R})d\lambda.
\end{equation}

\noindent We first deal with the integral (\ref{first}), 
\begin{equation}\label{parts}
\begin{split}
&\lim_{R\rightarrow\infty}\int_{Im(\lambda)=0}\frac{e^{i\lambda(|x|+|y|+t)}}{\lambda+i\mu}\varphi(\frac{\lambda^2}{R})d\lambda\\
=&\frac{1}{i(|x|+|y|+t)}\lim_{R\rightarrow\infty}\int_{Im(\lambda)=0}\frac{d}{d\lambda}(e^{i\lambda(|x|+|y|+t)})\frac{\varphi(\frac{\lambda^2}{R})}{\lambda+i\mu}d\lambda\\
=&\frac{-1}{i(|x|+|y|+t)}\lim_{R\rightarrow\infty}\int_{Im(\lambda)=0}e^{i\lambda(|x|+|y|+t)}\frac{d}{d\lambda}\big(\frac{\varphi(\frac{\lambda^2}{R})}{\lambda+i\mu}\big)d\lambda.
\end{split}
\end{equation}

\noindent The last equality is established by integration by parts. Apply the quotient rule then we get two integrals. The first one is 
\begin{equation*}
\begin{split}
\lim_{R\rightarrow\infty}\frac{2}{R}\int_{Im(\lambda)=0}\frac{\lambda\varphi'(\frac{\lambda^2}{R})e^{i\lambda(|x|+|y|+t)}}{\lambda+i\mu}d\lambda.
\end{split}
\end{equation*}

\noindent This limit is $0$, hence (\ref{parts}) equals the second integral obtained from the application of quotient rule, which is
\begin{equation*}
\begin{split}
&\frac{1}{i(|x|+|y|+t)}\lim_{R\rightarrow\infty}\int_{Im(\lambda)=0}\frac{e^{i\lambda(|x|+|y|+t)}\varphi(\frac{\lambda^2}{R})}{(\lambda+i\mu)^2}d\lambda\\
=&\frac{1}{i(|x|+|y|+t)}\int_{Im(\lambda)=0}\frac{e^{i\lambda(|x|+|y|+t)}\lim_{R\rightarrow\infty}\varphi(\frac{\lambda^2}{R})}{(\lambda+i\mu)^2}d\lambda\\
&\hspace{40mm}(\mbox{by Dominated Convergence Theorem})\\
=&\frac{1}{i(|x|+|y|+t)}\int_{Im(\lambda)=0}\frac{e^{i\lambda(|x|+|y|+t)}}{(\lambda+i\mu)^2}d\lambda.\\
\end{split}
\end{equation*}

\noindent Therefore term (\ref{first}) becomes 
\begin{equation}\label{fir}
\frac{-1}{8\pi^2 |x||y|(|x|+|y|+t)}\int_{Im(\lambda)=0}\frac{e^{i\lambda (|x|+|y|+t)}}{(\lambda+i\mu)^2}d\lambda.
\end{equation}

\noindent Now we can shift the contour $Im(\lambda)=0$ upwards to $Im(\lambda)=M$ for any $M>0$, 
and the integral should stay the same except when the contour moves across a pole. When $M\rightarrow\infty$, 
the integral approaches zero. 

\vspace{3mm}

\noindent Similarly, term (\ref{second}) becomes 
\begin{equation}\label{sec}
\frac{1}{8\pi^2|x||y|(|x|+|y|-t)}\int_{Im(\lambda)=0}\frac{e^{i\lambda (|x|+|y|-t)}}{(\lambda+i\mu)^2}d\lambda.
\end{equation}

\noindent 
In the region $t\leq |x|+|y|$, 
we can shift the contour $Im(\lambda)=0$ upwards, 
while in the region $t\geq |x|+|y|$, we can 
shift it downwards.
As before, the integral stays the same except when the contour moves across a pole, and 
the integral approaches zero when the contour is shifted further and further away.
Note that for $t=|x|+|y|$, we have a choice between shifting it upwards or downwards 
so we can always avoid the pole, hence we know the integral is $0$.
We continue the computation in two subcases:

\vspace{2mm}

\noindent {\bf Subcase 2(a): when $\mu>0$.}

\noindent In this case the pole, which is the resonance $a(\mu)=-i\mu$, lies on the negative imaginary axis. 
We shift the contour upwards for the integral (\ref{fir}), and the integral goes to $0$. 
While for the integral (\ref{sec}), it depends on the sign of $|x|+|y|-t$:
\begin{enumerate}
\item[\rm (i)] In the region $t\leq|x|+|y|$, the contour is also shifted upwards, so it also goes to $0$. Hence (\ref{concentrate}) equals $0$.
\item[\rm (ii)] In the region $t>|x|+|y|$, the contour is shifted downwards hence across the pole $-i\mu$.
The residue of the integrand at the pole $\lambda=-i\mu$ is
\[i(|x|+|y|-t)e^{\mu(|x|+|y|-t)},\]
since the winding number is $-1$, by Residue Theorem, we know (\ref{sec}), ie (\ref{concentrate}) equals
\[\frac{e^{\mu(|x|+|y|-t)}}{4\pi |x||y|}.\]
\end{enumerate}
\noindent Combine the two cases using a single expression, the kernel of $\frac{\sin(t\sqrt{L^\mu})}{\sqrt{L^\mu}}$ 
contributed by $K_{extra}$ for $\mu>0$ and $t\geq 0$ is
\begin{equation*}
\frac{H(t-|x|-|y|)e^{\mu(|x|+|y|-t)}}{4\pi |x||y|}.
\end{equation*}

\vspace{2mm}

\noindent {\bf Subcase 2(b): when $\mu<0$.}

\noindent In this case we have a negative eigenvalue and, 
\[\frac{\sin(t\sqrt{L^\mu})}{\sqrt{L^\mu}}
=\lim_{R\rightarrow\infty}\int_0^\infty\frac{\sin(t\sqrt{\sigma})}{\sqrt{\sigma}}\varphi(\frac{\sigma}{R})dP_\sigma
+\frac{\sin(i\mu t)}{i\mu}P_{-\mu^2}.\]
Let's first deal with the $\frac{\sin(i\mu t)}{i\mu}P_{-\mu^2}$ term. 
To do that we calculate the kernel of $P_{-\mu^2}$.
As discussed before, the eigenspace of $-\mu^2$ is one dimensional, and a normalised eigenfunction is 
\[v_\mu(x)=\sqrt{\frac{-\mu}{2\pi}}\frac{e^{\mu |x|}}{|x|},\]
so the kernel of $P_{-\mu^2}$ is 
\[v_\mu(x)\overline{v_\mu(y)}=\frac{-\mu e^{\mu (|x|+|y|)}}{2\pi|x||y|}.\]
Therefore the kernel of $\frac{\sin(i\mu t)}{i\mu}P_{-\mu^2}$ is
\begin{equation}\label{neg}
\frac{\sin(i\mu t)}{i\mu}\frac{-\mu e^{\mu (|x|+|y|)}}{2\pi|x||y|}
=\frac{\sinh(\mu t)}{\mu}\frac{-\mu e^{\mu (|x|+|y|)}}{2\pi|x||y|}
=\frac{-e^{\mu(|x|+|y|+t)}}{4\pi |x||y|}+\frac{e^{\mu(|x|+|y|-t)}}{4\pi |x||y|}.
\end{equation}

\noindent Now we calculate $\lim_{R\rightarrow\infty}\int_0^\infty\cos(t\sqrt{\sigma})\varphi(\frac{\sigma}{R})dP_\sigma$. 
The pole, ie the resonance $a(\mu)=-i\mu$, now lies on the positive imaginary axis, 
hence integral (\ref{fir}) may contribute some value.
The residue of the integrand at the pole $\lambda=-i\mu$ is
\[i(|x|+|y|+t)e^{\mu(|x|+|y|+t)},\]
since the winding number is $1$, by the Residue Theorem, the integral (\ref{fir}) equals
\begin{equation}\label{cancel}
\frac{e^{\mu(|x|+|y|+t)}}{4\pi |x||y|}.
\end{equation}

\noindent Again as before, to deal with (\ref{sec}) we split into two cases.

\begin{enumerate}
\item[\rm (i)] In the region $t<|x|+|y|$, the contour is shifted upwards, hence across the pole $-i\mu$. 
The residue of the integrand at the pole $\lambda=-i\mu$ is
\[i(|x|+|y|-t)e^{\mu(|x|+|y|-t)}.\]
The winding number is $1$, so by Residue Theorem, we know (\ref{sec}) equals
\begin{equation}\label{cancel2}
\frac{-e^{\mu(|x|+|y|-t)}}{4\pi|x||y|}.
\end{equation}

\noindent Adding (\ref{neg}), (\ref{cancel}) and (\ref{cancel2}) gives us $0$.

\item[\rm (ii)] In the region $t\geq |x|+|y|$, the contour is shifted downwards, so it doesn't move across any pole, 
so (\ref{sec}) is $0$ here. We add (\ref{neg}) and (\ref{cancel}) to get
\[\frac{e^{\mu(|x|+|y|-t)}}{4\pi |x||y|}.\]
\end{enumerate}
\noindent Combine above two cases into a single expression, the kernel of $\frac{\sin(t\sqrt{L^\mu})}{\sqrt{L^\mu}}$
contributed by $K_{extra}$ for $\mu<0$ and $t\geq 0$ is
\begin{equation*}
\frac{H(t-|x|-|y|)e^{\mu(|x|+|y|-t)}}{4\pi |x||y|}.
\end{equation*}

\noindent We have completed the proof.

\end{proof}

\begin{remark}
In the three cases of different signs of $\mu$, we have different expressions for the kernel for $t=|x|+|y|$. 
This is fine because we can ignore the value of the kernel on a set of measure $0$, but we do need to 
know that it is finite there so that we can be sure there is no distribution supported on this set.
\end{remark}

\end{document}